\documentclass[12pt,a4paper]{article}\usepackage{amsthm}\usepackage[dvipdfmx]{graphicx}
\usepackage{latexsym}
\usepackage{amssymb}
\usepackage{amsmath}
\usepackage{amscd}
\usepackage{amsfonts}
\usepackage{mathrsfs}
\usepackage{enumerate}
\usepackage[noadjust]{cite}
\usepackage{bm}
\usepackage[usenames]{color}
\usepackage[active]{srcltx}
\usepackage{ulem}
 \makeatletter
    
    \@addtoreset{equation}{section}
  \makeatother
\theoremstyle{plain}
\newtheorem{theorem}{Theorem}[section]

\newtheorem*{fact*}{Fact}
\newtheorem{lemma}[theorem]{Lemma}
\newtheorem{corollary}[theorem]{Corollary}
\theoremstyle{definition}
\newtheorem{definition}[theorem]{Definition}

\newtheorem{example}[theorem]{Example}
\newtheorem{fact}[theorem]{Fact}
\numberwithin{equation}{section}

\newcommand{\bx}{\boldsymbol{x}}

\newcommand{\be}{\boldsymbol{e}}
\newcommand{\e}{\boldsymbol{e}}

\newcommand{\bgamma}{\boldsymbol{\gamma}}

\newcommand{\inner}[2]{\left\langle{#1},{#2}\right\rangle}
\newcommand{\inners}[2]{\langle{#1},{#2}\rangle}

\newcommand{\vect}[1]{\boldsymbol{#1}}
\newcommand{\y}{\vect{y}}

\newcommand{\bv}{\vect{v}}
\newcommand{\bL}{\vect{L}}
\newcommand{\N}{\vect{N}}
\newcommand{\T}{\vect{T}}
\newcommand{\se}{_{\rm E}}

\newcommand{\rank}{\operatorname{rank}}
\newcommand{\Gr}{\operatorname{Gr}}
\newcommand{\ep}{\varepsilon}
\newcommand{\R}{{\mathbb R}}

\newcommand{\pmt}[1]{{\begin{pmatrix} #1  \end{pmatrix}}}

\begin{document}
\title{Geometry of 
lightlike locus on mixed type surfaces in
Lorentz-Minkowski 3-space from a contact viewpoint}
\author{Atsufumi Honda\thanks{Partially supported  by JSPS KAKENHI 19K14526},
Shyuichi Izumiya, Kentaro Saji\thanks{
Partially  supported  by JSPS KAKENHI 18K03301 and the Japan-Brazil
bilateral project JPJSBP1 20190103.}\\
 and\\
Keisuke Teramoto\thanks{Partially  supported  by JSPS KAKENHI 19K14533}}
\date{\today}

\maketitle

\begin{abstract}
A surface in the Lorentz-Minkowski $3$-space
is generally a mixed type surface, namely,
it has the lightlike locus.
We study local differential geometric
properties of such a locus on a mixed type surface.
We define a frame field along a lightlike locus, and
using it, we define two
lightlike ruled surfaces along a lightlike locus
which can be regarded as lightlike approximations of
the surface along the lightlike locus.
We study a relationship of singularities of these 
lightlike surfaces and differential geometric properties of
the lightlike locus.
We also consider the intersection curve of two
lightlike approximations, which gives a model curve of the lightlike locus.
\end{abstract}

\renewcommand{\thefootnote}{\fnsymbol{footnote}}

\section{Introduction}\label{sec:intro} 
Let $f:U\to\R^3_1$ be a $C^\infty$ immersion or a ``frontal''
from an open set $U$ in $\R^2$ into the $3$-dimensional Lorentz-Minkowski
space $\R^3_1$.
When $f$ is an immersion,
spacelike, lightlike and timelike points are defined by the usual way,
and it can be defined analogically for a frontal.
The notion of lightlike points is an independent notion 
from singular points of $f$.
In fact, it is a singular point of the induced metric.
In this paper, we assume that the set of the lightlike points $L(f)$ of $f$
is a curve, and the lightlike locus $f|_{L(f)}$ is a spacelike regular curve.
Under this assumption, we define a moving frame field along $L(f)$.
The frame consists of the tangent vector of $f|_{L(f)}$ and
the two lightlike vectors $\bL,\N$.
The vector $\bL$ is tangent to the surface at $f|_{L(f)}$,
and $\N$ is normal to the surface  in the Euclidean sense.
Using it, we construct two special lightlike ruled surfaces 
whose ruling directions are $\bL$ and $\N$.
These surfaces can be regarded as
lightlike approximations of $f$ along $L(f)$.
We give conditions that singularities of
these surfaces are cuspidal edges and 
swallowtails in terms of
certain geometric properties of $f$.
This kind of study is firstly given in \cite{IO}
for an immersion in $\R^3$ with a given curve on it.
In our case, different from the Euclidean case,
rulings are lightlike lines. If singularities of
a lightlike approximation is a constant point, then
it is a lightcone.
On the other hand, geometry and singularities of developable surface 
defined by a moving frame as a curve in $\R^3_1$ is
studied in \cite{izuframe}.
Our moving frame is associated to the immersion or frontal,
geometric meanings are related to the properties of
the curve as a curve on the original $f$, so that it is deeply depending on the properties of the differential geometry on
the frontal $f$.
Furthermore, we consider the pair of contacts of
two lightlike approximations with the lightlike locus.
Then the intersection curve of this pair can be considered as 
the model curve of the lightlike locus.
We introduce a new notion of
the contact orders of pairs.
In our case, the contact orders of pairs of lightlike approximations characterize the singularities 
of the lightlike approximations and the pedals of the lightlike locus.

\section{Preliminaries}\label{sec:pre} 
\subsection{Frontals in $\R^3_1$}\label{sec:frontal} 
Let $\R^3_1$ be the Lorentz-Minkowski $3$-space
equipped with the scalar product
$\inner{\bx}{\y}=-x_0y_0+x_1y_1+x_2y_2$, where
$\bx=(x_0,x_1,x_2)$, $\y=(y_0,y_1,y_2)$.
Let $\Gr(2,3)$ be the Grassmannian of $2$-planes in $\R^3_1$ and
consider a subbundle $\R^3_1\times \Gr(2,3)\subset \R^3_1\times\R^3=T\R^3_1$.
It
can be identified with the projective tangent bundle
$PT\R^3_1$ 
via
$PT_q\R^3_1 \ni V_q\mapsto (V_q)^\perp$ by 
the scalar product $\inner{~}{~}$,
where $\bx^\perp=\{\y \in \R^3_1\,|\,\inner{\bx}{\y}=0\}$.
It is well known that $PT\R^3_1=\R^3_1\times P\R^3$ is a contact manifold.
Let $U$ be a domain in $\R^2$.
A map $f:U\to\R^3_1$ is a {\it frontal}\/ if
there exists a map $F:U\to \R^3_1\times P\R^3$ of
the form $F=(f,[\nu])$ such that
$F$ is an isotropic map, namely, 
$\inner{df(X)}{\nu}=0$ for any $X\in T_pU$ and $p\in U$.
Such an $F$ is called an {\it isotropic lift\/}
of $f$.
Identifying $PT\R^3_1$ with $\R^3_1\times \Gr(2,3)$,
$F=(f,[\nu])$ can
be identified with $(f,\nu^\perp)$,
where $\nu^\perp:U\to \Gr(2,3)$ be a $2$-dimensional subspace-valued
map.
We call $\nu^\perp(p)$ the {\it limiting tangent plane\/} at $p\in U$.
A frontal is a {\it front\/} if the isotropic lift
$F:U\to \R^3_1\times P\R^3$ is an immersion.
Let $f:U\to \R^3_1$ be an immersion, and let $(u,v)$ be
a coordinate system on $U$.
We set
$[\nu]=[f_u\times f_v]$, where
$$
\bx\times \y=\det\pmt{
-\e_0&&\\
\e_1&\bx&\y\\
\e_2&&}
\quad \left(\e_0=\pmt{1\\0\\0},
\ \e_1=\pmt{0\\1\\0},
\ \e_2=\pmt{0\\0\\1}\right)
$$
is the cross product.
Then $f$ is a front.
This $[\nu]$ is called a {\it lightcone Gauss map\/} of 
$f$ \cite{peili,izulidev}.
Thus, an immersion is a front. On the other hand,
a frontal may have singular points.
A singular point of a frontal germ $f$ at $p$ 
is a {\it cuspidal edge\/} if it is 
${\cal A}$-equivalent
to $(u,v)\mapsto(u,v^2,v^3)$ at the origin,
where two map germs $f_i:(\R^2,p_i)\to(\R^2,q_i)$ $(i=1,2)$ 
are ${\cal A}$-{\it equivalent\/} if there exist
diffeomorphism germs $\phi:(\R^2,p_1)\to(\R^2,p_2)$
and $\Phi:(\R^3_1,q_1)\to(\R^3_1,q_2)$ such that
$\Phi\circ f\circ\phi^{-1}=g$ holds.
Cuspidal edges are the most fundamental singularities
appear on fronts.
The generic singularities of fronts are cuspidal edges and swallowtails.
A singular point of a frontal germ $f$ at $p$ 
is a {\it swallowtail\/}
if it is ${\cal A}$-equivalent
to $(u,v)\mapsto(u,4v^3+2uv,4v^4+uv^2)$ 
at the origin.
We denote by $S(f)$ 
the set of singular points.
In this decade, differential geometric properties of frontals in $\R^3$ and
Riemannian $3$-manifold are investigated by many authors
(see \cite{ISTe,MSUY,SUY} for example).

On the other hand, points on frontals in $\R^3_1$ can be classified 
into the following three cases.
A non-zero vector $\bx\in \R^3_1$ is said to be
{\it spacelike\/}
(respectively,\ {\it timelike\/}, {\it lightlike\/}),
if $\inner{\bx}{\bx}$ is positive
(respectively, negative, zero).
We denote by $LC^*$ the set of lightlike vectors.
A plane $P=\bx^\perp\subset\R^3_1$ is said to be
{\it spacelike\/}
(respectively,\ {\it timelike\/}, {\it lightlike\/}),
if $\bx$ is timelike,
(respectively, spacelike, lightlike).
Let $f$ be a frontal and $(f,[\nu])$ its isotropic lift.
A point $p\in U$ is said to be
{\it spacelike\/}
(respectively, {\it timelike\/}, {\it lightlike\/}) {\it point\/},
if $\nu^\perp(p)$
is spacelike
(respectively, timelike, lightlike).
We denote by $U_+$ 
(respectively, $U_-$, $L(f)$)
the set of spacelike 
(respectively,\ timelike, lightlike)
points.
If $U_+$, $U_-$ and $L(f)$ are all non-empty,
then $f$ is said to be {\it mixed type}.

There are many studies on differential geometric property of spacelike regular surfaces or spacelike frontals 
in $\R^3_1$ (see \cite{honizu,iz1,iz3,iz4,fuji3,UY_hokudai} for example), 
and surfaces whose lightlike point set $L(f)$ is non-empty
 (see \cite{Akamine,izulidev,fhkuy,fuji1,fuji2,
HKKUY,iz2,kly,kosskri,peili,remizov,remizovtari,tari,UY_hokudai,UY2018} for example). 
However, local behavior of a lightlike point
are investigated only on frontals with special curvature properties.

\subsection{Criteria for singularities of fronts in $\R^3_1$}\label{sec:cri}
To state criteria for singularities of fronts in $\R^3_1$,
we firstly recall criteria for singularities of fronts in $\R^3$
given in \cite{krsuy}.
Let $f:U\to\R^3$ be a frontal and $F=(f,[\nu\se]):U\to PT\R^3=\R^3\times P\R^3$ 
its isotropic lift, namely, $\nu\se$ satisfies 
$df_p(X)\cdot\nu\se(p)=0$ for any $p\in U$ and $X\in T_pU$,
where the dot ``$\cdot$'' stands for the Euclidean inner product.
We set $\lambda\se=\det(f_u,f_v,\nu\se)$.
Then $S(f)=\lambda\se^{-1}(0)$ holds.
A function $\Lambda\se$ is called 
an {\it identifier of singularities\/} if it is a non-zero functional
multiple of $\lambda\se$. Let $\Lambda\se$ be an identifier of singularities.
Let $p\in U$ be a singular point satisfying $\rank df_p=1$.
Then there exists a non-zero vector field $\eta$ on a neighborhood
of $p$ such that
the kernel of $df_q$ is generated by $\eta(q)$ for any
$q\in S(f)$. We call it a {\it null vector field}.
The following fact holds.
\begin{fact}\label{fact:crie}{\rm (\cite[Corollary 2.5]{suy3})}
Let\/ $f:U\to\R^3$ be a frontal, and let\/ $p\in U$ 
be a singular point satisfying\/ $\rank df_p=1$.
Under the above notation, 
\begin{enumerate}
\item\label{item:ce} $f$ at\/ $p$ is a cuspidal edge if and only if\/
$f$ is a front at\/ $p$, and\/ $\eta\Lambda\se(p)\ne0$.
\item\label{item:sw} $f$ at\/ $p$ is a swallowtail if and only if\/
$f$ is a front at\/ $p$, and\/ $\eta\Lambda\se(p)=0$,
$\eta\eta\Lambda\se(p)\ne0$, $d(\Lambda\se)_p\ne0$.
\end{enumerate}
\end{fact}
A similar criteria hold for frontals in $\R^3_1$ by
a slight modification.
Let $f:U\to\R^3_1$ be a frontal and $F=(f,[\nu])$ its isotropic lift,
where $[\nu]\in P\R^3_1$, namely, $\nu$ satisfies 
$\inner{df_p(X)}{\nu(p)}=0$ for any $p\in U$ and $X\in T_pU$.
Taking a vector field $\T$ along $f$
such that $\T$ is transverse to $\nu^\perp$,
we set $\Lambda=\det(f_u,f_v,\T)$.
Then we see that
$\Lambda$ is a non-zero functional multiple of
$\lambda\se$.
Thus $\Lambda$ is an identifier of singularities.
By using $\Lambda$, we can recognize
whether $f$ at $p$ is a cuspidal edge or a swallowtail.

\subsection{Discriminant sets of functions}\label{sec:discr} 
The lightlike approximations are envelopes of
lightlike planes along the lightlike locus.
To construct lightlike approximation,
we use the theory of unfolding of a function and
discriminant set which can describe the envelopes.

Let $a:(\R,0)\to(\R,0)$ be a function. 
For a manifold $X$ and $p\in X$,
a function $A:(\R\times X,(0,p))\to(\R,0)$ 
is called an {\it unfolding\/} of $a$ if $A(u,p)=a(u)$ holds.
In this setting, we regard $A$ as a parameter family of a function $a$.
We assume that  $a'(0)=0$ $('=\partial/\partial u)$
and define the set $\Sigma_A$ and the 
{\it discriminant set\/} ${\cal D}_A$ of $A$ as
\begin{align*}
\Sigma_A&=
\{(u,q)\in \R\times X\,|\,A(u,q)=A_u(u,q)=0\},\\
{\cal D}_A&=
\{q\in X \ |\ \text{ there exists }u\in \R\text{ such that } 
A(u,q)=A_u(u,q)=0\}.
\end{align*}
If the map $(A,A_u)$ is a submersion at $(0,p)$, then
$\Sigma_A$ is a manifold.
By definition,  the discriminant set is the envelope of the family 
$\{q\in X | A(u,q)=0\}_{u\in \R}$.
See \cite[Section 7]{Bru-Gib} or \cite[Section 5]{irrt} 
for the general theory of  unfoldings and their  discriminant sets.

\section{Lightlike surfaces}\label{sec:lldev}
\subsection{Frame along lightlike locus}\label{sec:frame} 
Let $f:U\to \R^3_1$ be a frontal 
whose lightlike point set $L(f)$ is non-empty,
and let $F=(f,[\nu])$ be its isotropic lift.
We take $p\in L(f)$, and assume that 
\begin{equation}\label{eq:assump0}
\begin{array}{l}
L(f)\text{ is a regular curve in }U\text{ which is}\\
\hspace{15mm}
\text{parametrized by }\gamma:(-\ep,\ep)\to U\text{ near $p=\gamma(0)$.}
\end{array}
\end{equation} 
Under this assumption,
$f(L(f))$ is called the {\it lightlike locus\/}.
We set $\hat\bgamma=f\circ\gamma$.
We assume that 
\begin{equation}\label{eq:assump1}
\text{if $p\in S(f)$, then $\gamma'(0)\not\in \ker df_p$, }
\end{equation}
where $'=d/du$.
This implies that $\hat\bgamma$ is a regular curve in $\R^3_1$.
Furthermore, we also assume that 
\begin{equation}\label{eq:assump2}
\text{$\hat\bgamma'(u)$ is not lightlike.}
\end{equation}
By these assumptions, 
$\hat\bgamma$ is a spacelike regular curve in $\R^3_1$.
A frontal which satisfies the assumptions
\eqref{eq:assump0}, \eqref{eq:assump1} and \eqref{eq:assump2}
are said to be an {\it admissible frontal.}

A frame along a spacelike regular curve 
in $\R^3_1$ is obtained in \cite{izuframe}.
Here we consider a frame along $L(f)$ as a curve
on the surface $f$.
Then we can take a parameter $u$ such that
$|\hat\bgamma'(u)|=1$, where
$|\bx|=\sqrt{|\inner{\bx}{\bx}|}$.
Setting $\e(u)=\hat\bgamma'(u)$, 
we have a frame $\{\e,\bL,\N\}$ along
$\hat\bgamma(u)$ satisfying
\begin{equation}\label{eq:condln}
\inner{\e}{\e}=1,\quad
\inner{\e}{\bL}=0,\quad
\inner{\e}{\N}=0,\quad
\inner{\bL}{\bL}=0,\quad
\inner{\bL}{\N}=1,\quad
\inner{\N}{\N}=0,
\end{equation}
and the plane spanned by $\e$ and $\bL$ is 
the limiting tangent plane of
$f$ along $\hat\bgamma$.
The limiting tangent plane is also called the
{\it osculating plane\/}, and the plane
spanned by $\e$ and $\N$ is called the {\it transversal lightlike plane\/}
of $f$ along $\hat\bgamma$,
since it is a unique lightlike plane which contains $\e$ and
transverse to $f$.
The following Frenet-Serret type formula
$$
(\e,\bL,\N)'=(\e,\bL,\N)\pmt{
0&-\alpha_L&-\alpha_N\\
\alpha_N&-\alpha_G&0\\
\alpha_L&0&\alpha_G}
$$
holds,
where
\begin{equation}\label{eq:alphalng}
\alpha_L(u) = \inner{ \hat{\bgamma}''(u) }{ \bL(u) },\quad
\alpha_N(u) = \inner{\hat{\bgamma}''(u)}{\N(u)},\quad
\alpha_G(u) = \inner{\bL(u)}{\N'(u)}.
\end{equation}
These three functions are determined by
$f$ and $\bL$.
We set $\overline{\bL}(u)=\psi(u)\bL(u)$, 
where $\psi$ is a never vanishing
function.
Then setting $\overline{\N}=\N/\psi$,
the frame $\{\e,\overline{\bL},\overline{\N}\}$
satisfies the condition \eqref{eq:condln}, and
it holds that
\begin{equation}\label{eq:alphadepend}
\overline{\alpha_L}=\psi\alpha_L,\quad
\overline{\alpha_N}=\alpha_N/\psi,\quad
\overline{\alpha_G}=\alpha_G+\psi(\psi^{-1})'
\end{equation}
where $\overline{\alpha_L},\overline{\alpha_N},\overline{\alpha_G}$
are defined by \eqref{eq:alphalng} with respect to the frame
$\{\e,\overline{\bL},\overline{\N}\}$.

Let $f:U\to \R^3_1$ be an immersion
with non-empty lightlike point set $L(f)$.
Suppose that $L(f)$ consists of lightlike points of the first kind
(see \cite[Definition 2.2]{HST} for details).
Then $f$ is a mixed type surface,
is admissible in the above sense,
and $\alpha_L$, $\alpha_N$, $\alpha_G$ 
can be regarded as invariants of $f$ as follows:
There exists a vector field $l$ on $U$
such that $df(l(q))=\bL(q)$ for any $q\in L(f)$
and $\beta=l\inner{df(l)}{df(l)}|_{L(f)}$ 
does not vanish along $L(f)$.
Let $\kappa_L$, $\kappa_N$, and $\kappa_G$
be the {\it lightlike singular curvature}, 
the {\it lightlike normal curvature}, and
the {\it lightlike geodesic torsion\/}
of $f$ along $L(f)$, respectively\footnote{These three invariants 
$\kappa_L$, $\kappa_N$, $\kappa_G$
are introduced in \cite{HST} to investigate the behavior 
of the Gaussian curvature of mixed type surfaces at lightlike points,
cf. \cite[Theorem B]{HST}.}
(\cite[Definition 3.2]{HST}).
By \cite[Proposition 3.5]{HST}, it holds that
\begin{equation}\label{eq:frenet2}
\begin{array}{rl}
\alpha_L=\beta^{1/3}\kappa_L,\quad
\alpha_N=\beta^{-1/3}\kappa_N,\quad
\alpha_G=\kappa_G+\beta^{1/3}(\beta^{-1/3})'.
\end{array}
\end{equation}
One can take $l$ satisfying $\beta=1$,
by rechoosing $\beta^{-1/3} \,l$ instead of $l$.
Then
$\alpha_L=\kappa_L$,
$\alpha_N=\kappa_N$,
$\alpha_G=\kappa_G$ hold.

\subsection{Osculating and transversal lightlike surfaces}
\label{sec:construction}
Let $f:U\to \R^3_1$ be an admissible frontal. 
Under the notation in Section \ref{sec:frame}, 
we consider the discriminant sets of
the following functions
\begin{align}
H_L(u,\bx)=&\inner{\bx-\hat\bgamma(u)}{\bL(u)}:L(f)\times \R^3_1\to\R,\nonumber\\
H_N(u,\bx)=&\inner{\bx-\hat\bgamma(u)}{\N(u)}:L(f)\times \R^3_1\to\R,\nonumber\\
G(u,\bx)=&\inner{\bx-\hat\bgamma(u)}{\bx-\hat\bgamma(u)}:L(f)\times \R^3_1\to\R,
\label{eq:funcg}\\
\widetilde{H}(u,\widetilde{\bv},r)=
&\inner{\hat\bgamma(u)}{\widetilde{\bv}}-r:L(f)\times S^1_+\times\R^\times\to\R,
\label{eq:functh}
\end{align}
where
$S^1_+=\{(1,x_1,x_2)\,|\,(x_1)^2+(x_2)^2=1\}$
and
$\R^\times=\R\setminus\{0\}$.
Here we regard $(X,p)=(\R^3_1,p)$ for $p\in \R^3_1$ and 
$A=H_L, H_N,G,\widetilde{H}$ under the notation in Section \ref{sec:discr}.
Since $\bL$ and $\N$ are lightlike, 
the discriminant set ${\cal D}_{H_L}$ is 
the  envelope of families of lightlike planes which 
are tangent to the frontal at $\hat\bgamma(u)$,
and
${\cal D}_{H_N}$ is 
the  envelope of families of transversal lightlike planes 
of $f$ at $\hat\bgamma(u)$.
We calculate the discriminant set of $H_L$ (respectively, $H_N$)
under the assumption $\alpha_L\ne0$ (respectively, $\alpha_N\ne0$).
Let us assume $\alpha_L\ne0$ (respectively, $\alpha_N\ne0$) for any $u\in I$.
Since 
$$
H_L'=\inner{\bx-\hat\bgamma(u)}{-\alpha_L\e-\alpha_G\bL}\quad
\text{and}
\quad
H_N'=\inner{\bx-\hat\bgamma(u)}{-\alpha_N\e + \alpha_G\N},$$
the discriminant sets ${\cal D}_{H_L}$, ${\cal D}_{H_N}$ 
can be parametrized by
$$
f_L(u,v)=\hat\bgamma(u)+v\bL(u)\quad\text{and}\quad
f_N(u,v)=\hat\bgamma(u)+v\N(u)
$$
respectively, where $'=\partial/\partial u$.
Since 
$\inner{f_L'}{\bL}=\inner{(f_L)_v}{\bL}=0$
and 
$\inner{f_N'}{\N}=\inner{(f_N)_v}{\N}=0$
hold, 
$\bL$ and $\N$ are lightcone Gauss map of $f_L$ and
$f_N$ respectively even on the set of singular points.
Since each lightcone Gauss map degenerates to a curve,
$f_L$ and $f_N$ have both zero Gaussian curvature 
in the Euclidean sense
(\cite[Theorem 3.1]{izulidev}).
Moreover, the limiting tangent plane of $f_L$ 
coincides with that of $f$ along $L(f)$,
and the limiting tangent plane of $f_N$ 
is the transversal lightlike plane of $f$ along $L(f)$.
In this sense, we call 
$f_L(u,v)$ (respectively, $f_N(u,v)$) the 
{\it osculating lightlike surface\/} 
(respectively, {\it transversal lightlike surface\/})
of $f$ along $L(f)$.
Now, let us investigate $f_L$ (respectively, $f_N$), 
without assuming $\alpha_L\ne0$ (respectively, $\alpha_N\ne0$), 
since such the assumptions
are not necessary for the definition of 
$f_L$ and $f_N$.
After obtaining these surfaces, it does not necessary
the condition $\alpha_L\ne0$ (respectively, $\alpha_N\ne0$).
We investigate these surfaces 
without the condition $\alpha_L\ne0$
(respectively, $\alpha_N\ne0$).

By a similar calculation,
we see ${\cal D}_{G}={\cal D}_{H_L}\cup{\cal D}_{H_N}$.
On the other hand, since $\widetilde{H}'=\inner{\be}{\widetilde{\bv}}$ holds,
$\widetilde{H}'=0$ implies $\widetilde{\bv}=a\bL+b\N$.
Since $\widetilde{\bv}$ is lightlike, $ab=0$.
Moreover, {$\widetilde{\bv}\in S^1_+$}, we have
$\widetilde{\bv}=\widetilde{\bL}$ or $\widetilde{\bv}=\widetilde{\N}$.
Thus
$$
{\cal D}_{\widetilde{H}}=
\Big\{\Big(\widetilde{\bL(u)},
\inner{\hat\bgamma(u)}{\widetilde{\bL(u)}}\Big)\,\Big|\,
u\in L(f)\Big\}\cup
\Big\{\Big(\widetilde{\N(u)},
\inner{\hat\bgamma(u)}{\widetilde{\N(u)}}\Big)\,\Big|\,
u\in L(f)\Big\}.
$$
Let $\Phi:S^1_+\times\R^\times\to LC^*$ be the diffeomorphism
$\Phi(\widetilde{\bv},r)= r\widetilde{\bv}$.
We set
$$LP_L(u)=\inners{\hat\bgamma(u)}{\widetilde{\bL(u)}}\widetilde{\bL(u)}
\quad
\text{and}
\quad
LP_N(u)=\inners{\hat\bgamma(u)}{\widetilde{\N(u)}}\widetilde{\N(u)}.
$$
The curve $LP_L$
(respectively, $LP_N$)
is called the {\it osculating lightcone pedal\/}
(respectively, 
{\it transversal lightcone pedal\/}) of $f$ (cf.\ \cite{izuframe}).
The union of the images of osculating lightcone pedal and
transversal lightcone pedal coincides with
$\Phi({\cal D}_{\widetilde{H}})$.
We also consider discriminant sets ${\cal D}_{G}$ and
${\cal D}_{\widetilde{H}}$ in the context of bi-contact of lightcones
in Section \ref{sec:contactinter}.

\subsection{Singularities of osculating and transversal lightlike surfaces}
\label{sec:lldevo}
We set two functions
$$
\sigma_L(u)=\alpha_L'(u)+\alpha_L(u)\alpha_G(u),\quad\text{and}\quad
\sigma_N(u)=\alpha_N'(u)-\alpha_N(u)\alpha_G(u).
$$
We have the following:
\begin{theorem}\label{thm:singsurf}
Let\/ $f:U\to \R^3_1$ be an admissible frontal.
Then $f_L$ and\/ $f_N$ are front.
A point\/ $(u,v)$ is a singular point of\/ $f_L$ $($respectively, $f_N)$
if and only if\/ $1-v\alpha_L=0$ $($respectively, $1-v\alpha_N=0)$.
A singular point\/ $(u,1/\alpha_L)$ of\/ $f_L$ 
$($respectively, $(u,1/\alpha_N)$ of\/ $f_N)$
is 
\begin{enumerate}
\item a cuspidal edge if and only if\/
$\sigma_L\ne0$ 
$($respectively, $\sigma_N\ne0)$ at\/ $u$.
\item a swallowtail
if and only if\/
$\sigma_L=0$ and\/ 
$\sigma_L'\ne0$
$($respectively, $\sigma_N=0$
and\/
$\sigma_N'\ne0)$ at $u$.
\end{enumerate}
\end{theorem}
We remark that although $\alpha_L,\alpha_N,\alpha_G$ do depend
on the choice of $\bL$, 
the positions of singular points of $f_L$ and $f_N$ do not depend
on it, and
the conditions (1) and (2) do not
depend on it (cf. \eqref{eq:alphadepend}).
In fact, if $\overline{\bL}=\psi\bL$, then $\overline{\N}=\N/\psi$,
and \eqref{eq:alphadepend} holds. Moreover, 
by \eqref{eq:alphadepend}, we have
\begin{equation}\label{eq:sigmaindep}
\overline{\alpha_L}'+\overline{\alpha_L}\,\overline{\alpha_G}=
\psi(\alpha_L'+\alpha_L\alpha_G),\quad
\overline{\alpha_N}'-\overline{\alpha_N}\,\overline{\alpha_G}=
\psi^{-1}(\alpha_N'-\alpha_N\alpha_G).
\end{equation}
\begin{proof}
We set $g(u,v)=f_L(u,v)$ for simplicity.
The isotropic lift of $f_L$ is $(f_L,[\bL])$.
Since $g'=(1-v\alpha_L)\e-v\alpha_G\bL$, and $g_v=\bL$,
setting $\Lambda=1-v\alpha_L$, $\Lambda$ can be taken as 
an identifier of singularities.
A null vector field $\eta$ is
$
\eta=\partial_u+v\alpha_G \partial_v
$. 
Since $\eta[\bL]\ne0$ if and only if $\bL$ and $\eta \bL$ are
linearly independent, $g$ at a singular point is a front
if and only if $\alpha_L\ne0$.
If $\alpha_L(u_0)=0$, then $(u_0,v)$ is not a singular point.
Thus $g$ is a front at any singular point.
We have $\eta\Lambda=-v\sigma_L$ and
$\eta\eta\Lambda=\sigma_L'$.
Thus the assertion for $g$ holds.
The assertion for $f_N$ can be shown by the same way
taking an isotropic lift $(f_N,[\N])$, 
an identifier of singularities $\Lambda=1-v\alpha_N$ and 
a null vector field
$
\eta=\partial_u-v\alpha_G \partial_v
$. 
\end{proof}
If $f$ is an immersion
and $L(f)$ consists 
of lightlike points of the first kind,
then the conditions in Theorem \ref{thm:singsurf}
can be stated in terms of the curvatures $\kappa_L,\kappa_N,\kappa_G$
as follows.
\begin{corollary}\label{cor:singsurf}
Let\/ $f:U\to \R^3_1$ be an immersion 
with non-empty lightlike point set $L(f)$.
Suppose that $L(f)$ consists of lightlike points of the first kind.
Then $f_L$ and\/ $f_N$ are front.
A point\/ $(u,v)$ is a singular point of\/ $f_L$ $($respectively, $f_N)$
if and only if\/ $1-v\kappa_L=0$ $($respectively, $1-v\kappa_N=0)$.
A singular point\/ $(u,1/\kappa_L)$ of\/ $f_L$ 
$($respectively, $(u,1/\kappa_N)$ of\/ $f_N)$
is 
\begin{enumerate}
\item a cuspidal edge if and only if\/
$\widetilde\sigma_L\ne0$ 
$($respectively, $\widetilde\sigma_N\ne0)$ at\/ $u$.
\item a swallowtail
if and only if\/
$\widetilde\sigma_L=0$ and\/ 
$\widetilde\sigma_L'\ne0$
$($respectively, $\widetilde\sigma_N=0$
and\/
$\widetilde\sigma_N'\ne0)$ at $u$.
\end{enumerate}
Here, 
$$
\widetilde\sigma_L(u)=\kappa_L'(u)+\kappa_L(u)\kappa_G(u),\quad\text{and}\quad
\widetilde\sigma_N(u)=\kappa_N'(u)-\kappa_N(u)\kappa_G(u).
$$
\end{corollary}
The functions $\sigma_L$ and
$\sigma_N$ correspond to the invariants
$k'\mp k\tau$ in \cite{izuframe} which play an important role
in their paper.
In \cite{izuframe}, they use the Frenet-Serret type frame,
this is an invariant of the curve in $\R^3_1$.

We can state the condition of lightcone pedal curves
in terms of $\alpha_L$ and $\alpha_N$.

\begin{theorem}\label{thm:singpedal}
Under the same setting in Theorem\/ {\rm \ref{thm:singsurf}},
the point\/ $\gamma(0)$ is a singular point of\/ 
the osculating lightcone pedal
$LP_L$ 
$($respectively, 
the transversal lightcone pedal $LP_N)$ 
if and only if\/
$\alpha_L(0)=0$
$($respectively, $\alpha_N(0)=0)$.
Moreover, $LP_L$ $($respectively, $LP_N)$
has a cusp at\/ $\gamma(0)$ if and only if\/
$\alpha_L(0)=0$, $\alpha_L'(0)\ne0$
$($respectively, $\alpha_N(0)=0$, $\alpha_N'(0)\ne0)$.
\end{theorem}
Like as we remarked just after Theorem \ref{thm:singsurf},
the conditions in this theorem do not
depend on the choice of $\bL$.
\begin{proof}
We see the function $\widetilde{h}(u)=\widetilde{H}(u,\widetilde{\bv}_0,r_0)$,
where
$r_0=\inners{\hat\bgamma(0)}{\widetilde{\bL}(0)}$,
$\widetilde{\bv}_0=\widetilde{\bL}(0)$
satisfies
\begin{align*}
\widetilde{h}'(0)&=\inners{\e}{\widetilde{\bv}_0}=0,\\
\widetilde{h}''(0)&=
\inners{\alpha_N\bL+\alpha_L\N}{\widetilde{\bv}_0}=\alpha_L(0)/L_0(0)
\quad(\bL=(L_0,L_1,L_2)),\\
\widetilde{h}'''(0)&=\inners{-2\alpha_L\alpha_N\e-\sigma_N\bL+\sigma_L\N}
{\widetilde{\bv}_0}=\sigma_L(0)/L_0(0).
\end{align*}
Taking a parametrization $\theta$ of $S^1_+$ by 
$\widetilde{\bv}=(1,\cos\theta,\sin\theta)$, then
$\widetilde{H}'_\theta(u,\widetilde{\bv},r)$ is the first component of
$\be\times \widetilde{\bv}$.
Then we see 
$\widetilde{h}''(0)=0$ if and only if $\alpha_L(0)=0$.
If $\widetilde{h}''(0)=0$, then
$\widetilde{h}'''(u)=0$ if and only if $\alpha_L'(0)=0$.
Furthermore, 
if $\widetilde{h}''(0)=0$ and 
$\widetilde{h}'''(u)\ne0$, then 
$\widetilde{H}(u,\widetilde{\bv},r)$ is a versal unfolding of
$\widetilde{h}$.
In fact, we assume 
the first component of $\be\times \widetilde{\bv}$ is zero,
then $\be\times \widetilde{\bv}$ is lightlike,
$e_0^2+e_1^2+e_2^2-2e_0(e_1\cos\theta+e_2\sin\theta)=0$
for $\e=(e_0,e_1,e_2)$.
On the other hand, since $\inners{\e}{\widetilde L}=0$,
we have $-e_0+e_1\cos\theta+e_2\sin\theta=0$.
This implies that $-e_0^2+e_1^2+e_2^2=0$, a contradiction.
By the well-known fact of the versal unfolding of
a function and its discriminant set (see \cite{Bru-Gib} for example),
if $\alpha_L\ne0$ (respectively, $\alpha_L=0$, $\alpha_L'\ne0$)
at $0$ then
${\cal D}_{\widetilde{H}}$ is locally diffeomorphic to
a regular curve (respectively, a $3/2$-cusp)
at $\big(\widetilde{\N}(0),\inners{\hat\bgamma(0)}{\widetilde{\N}(0)}\big)$.
By the diffeomorphism $\Phi:S^1_+\times\R^\times\to LC^*$,
the discriminant set
${\cal D}_{\widetilde{H}}$ is sent to 
the union of the images of $LP_L(u)$ and $LP_N(u)$.
Since the diffeomorphicity of the images implies the ${\cal A}$-equivalence
and $\widetilde{\bL}$ and $\widetilde{\N}$ are linearly independent,
we have the assertion.
\end{proof}
Like as Corollary \ref{cor:singsurf},
if $f$ is an immersion
and $L(f)$ consists 
of lightlike points of the first kind,
then the conditions in Theorem \ref{thm:singpedal}
can be stated in terms of $\kappa_L,\kappa_N,\kappa_G$
as follows.
\begin{corollary}\label{cor:singpedal}
Under the same setting in Corollary\/ {\rm \ref{cor:singsurf}},
the point\/ $\gamma(0)$ is a singular point of
the osculating lightcone pedal\/
$LP_L$ 
$($respectively, 
the transversal lightcone pedal\/ $LP_N)$ 
if and only if\/
$\kappa_L(0)=0$
$($respectively, $\kappa_N(0)=0)$.
Moreover, $LP_L$ $($respectively, $LP_N)$
has a cusp at\/ $\gamma(0)$ if and only if\/
$\kappa_L(0)=0$, $\kappa_L'(0)\ne0$
$($respectively, $\kappa_N(0)=0$, $\kappa_N'(0)\ne0)$.
\end{corollary}

\subsection{Contact of $L(f)$ with intersection curves}\label{sec:contactinter}
In this section, we study a lightlike locus by considering
contact of intersection curves of two model surfaces
defined by the lightlike locus with the frame defined in Section \ref{sec:frame}.
Let $f:U\to \R^3_1$ be an admissible frontal.
Regarding the discriminants of $G,\widetilde H$ 
in \eqref{eq:funcg}, \eqref{eq:functh},
under the notation in Section \ref{sec:frame}, 
we consider the following functions
\begin{align*}
G_L(\bx)&=\inner{\bx-\bx_L}{\bx-\bx_L},\quad
G_N(\bx)=\inner{\bx-\bx_N}{\bx-\bx_N},\\
\widetilde{H}_L(\bx)&=
\inner{\bx-\hat\bgamma(0)}{\widetilde{\bL(0)}},\quad
\widetilde{H}_N(\bx)=
\inner{\bx-\hat\bgamma(0)}{\widetilde{\N(0)}}
\end{align*}
from $\R^3_1$ into $\R$, where 
$\bx_L=\hat\bgamma(0)+\dfrac{\bL(0)}{\alpha_L(0)}$,
$\bx_N=\hat\bgamma(0)+\dfrac{\N(0)}{\alpha_N(0)}$
and
$\widetilde\bv=\bv/v_0$ for $\bv=(v_0,v_1,v_2)$.
We also remark here that $\bx_L$ and $\bx_N$ do not depend on the choice of
$\bL$ (cf. \eqref{eq:alphadepend}).
Then $G_L^{-1}(0)$ (respectively, $G_N^{-1}(0)$)
is the lightcone with the vertex $\bx_L$ (respectively, $\bx_N$),
and 
$\widetilde{H}_L^{-1}(0)$ (respectively, $\widetilde{H}_N^{-1}(0)$) 
is the lightlike plane
with the lightlike normal vector $\widetilde{\bL(0)}$
(respectively, $\widetilde{\N(0)}$) 
passing through $\hat\bgamma(0)$.
We call $G_L^{-1}(0)$, 
(respectively, $G_N^{-1}(0)$, $\widetilde{H}_L^{-1}(0)$,
$\widetilde{H}_N^{-1}(0)$)
the {\it osculating contact lightcone\/}
(respectively,
{\it transversal contact lightcone\/},
{\it osculating contact lightlike plane\/}
{\it transversal contact lightlike plane\/}) at $p$.
We consider the intersections of two of them.
The intersection of the
osculating contact lightcone and the transversal contact lightcone,
is a spacelike ellipse tangent to $L(f)$ at $p$, which is called an {\it osculating ellipse}.
The intersection of the
osculating contact  lightlike plane  and the transversal 
contact lightlike plane,
is a spacelike line tangent to $L(f)$ at $p$, which is the tangent line.
The intersection of the
osculating contact lightcone and the transversal contact lightlike plane,
(respectively, the osculating contact lightlike plane and 
the transversal contact lightcone)
is a spacelike parabola tangent to $L(f)$ at $p$, which is called an {\it $N$-osculating} (respectively, an {\it $L$-osculating}) parabola.
Since these curves are intersections of fundamental objects
in the Lorentz-Minkowski $3$-space,
they can be regarded as model curves of a lightlike locus of
a surface.
In fact, we can interpret that the spacelike ellipse represents how $L(f)$ looks round and
the line represents how $L(f)$ looks flat.
Moreover, the spacelike parabola represents how $L(f)$ looks semi-flat with respect to $\bL$ or $\N$.

We set
$$
g_L(u)=G_L(\hat\bgamma(u)),\ 
g_N(u)=G_N(\hat\bgamma(u)),\ 
\widetilde h_L(u)=\widetilde H_L(\hat\bgamma(u)),\ 
\widetilde h_N(u)=\widetilde H_N(\hat\bgamma(u)).
$$
\begin{definition}
Let $F:(\R^3_1,\hat\bgamma(0))\to(\R,0)$ be a function.
Then $F^{-1}(0)$ and $\hat\bgamma$ have a $k$-{\it point contact\/} at $u=0$
if $f=F\circ\hat\bgamma$ satisfies 
$f'=\cdots=f^{(k+1)}=0$, $f^{(k+2)}\ne0$ at $u_0$.
Let $F_j:(\R^3_1,\hat\bgamma(0))\to(\R,0)$ $(j=1,2)$ be two functions.
Then $F_1^{-1}(0)\cap F_2^{-1}(0)$ and
$\hat\bgamma$ have a $(k_1,k_2)$-{\it point contact\/} 
if $F_i^{-1}(0)$ and $\hat\bgamma$ have a $k_i$-point contact at $u=0$
for $i=1,2$.
\end{definition}
Applying this to $f=g_L,g_N,\widetilde h_L,\widetilde h_N$,
we see that
 the contact with the lightcone with the vertex $\bx_L$ (respectively, $\bx_N$)
can be
measured by $\sigma_L$ (respectively, $\sigma_N$),
and
the contact with the lightlike plane 
with the lightlike normal vector $\widetilde{\bL(0)}$ 
(respectively, $\widetilde{\N(0)}$),
passing through $\hat\bgamma(0)$
can be
measured by the invariant $\alpha_L$ (respectively, $\alpha_N$).
More precisely, the following theorem holds. 
\begin{theorem}\label{thm:ghcont}
{\rm (1)} 
The function\/ $g_L$ satisfies\/ $g_L=g_L'=\cdots=g_L^{(k)}=0$ 
at\/ $0$ if and only if\/
$\sigma_L=\sigma_L'=\cdots=\sigma_L^{(k-3)}=0$ at\/ $0$ for any\/ $k\geq3$.
Similarly, the function\/ $g_N$ satisfies\/ 
$g_N=g_N'=\cdots=g_N^{(k)}=0$ at\/ $0$ if and only if\/
$\sigma_N=\sigma_N'=\cdots=\sigma_N^{(k-3)}=0$
at\/ $0$ for any\/ $k\geq3$.

{\rm (2)} The function\/ $\widetilde h_L$ satisfies\/ 
$\widetilde h_L=\widetilde h_L'=\cdots=\widetilde h_L^{(k)}=0$ at\/ 
$0$ if and only if\/
$\alpha_L=\alpha_L'=\cdots=\alpha_L^{(k-2)}=0$ at\/ $0$ for any\/ $k\geq2$.
Similarly the function\/ $\widetilde h_N$ satisfies\/
$\widetilde h_N=\widetilde h_N'=\cdots=\widetilde h_N^{(k)}=0$
at\/ $0$ if and only if\/
$\alpha_N=\alpha_N'=\cdots=\alpha_N^{(k-2)}=0$ at\/ $0$ for any\/ $k\geq2$.
\end{theorem}
As we mentioned in \eqref{eq:sigmaindep}, the condition
$\sigma_L=\sigma_L'=\cdots=\sigma_L^{(k-3)}=0$ 
(respectively, $\sigma_N=\sigma_N'=\cdots=\sigma_N^{(k-3)}=0$) 
at\/ $0$ for any\/ $k\geq3$ does not depend on the choice of $\bL$.
Proof of this theorem is not difficult, but
complicated a little, we give a proof of this 
theorem in Appendix \ref{sec:proofthm}.
In \cite{izuframe}, a similar consideration by using
the Frenet frame along a curve in $\R^3_1$ is given.
See \cite[Propositions 2.1 and 2.2]{izuframe}.

By Theorem \ref{thm:ghcont} together with
Theorems \ref{thm:singsurf} and \ref{thm:singpedal},
we can conclude the contact of model curves
can be measured by the singularities 
of lightlike surfaces and lightcone pedals.
For the sake of simplified description,
we set the following terminology.
A frontal $f:(\R^2,p)\to\R^3$ has an $A_2$-{\it point\/} 
(respectively, $A_3$-{\it point\/}) at $p$ if
$f$ at $p$ is cuspidal edge (respectively, swallowtail).
A frontal $c:(\R,p)\to\R^2$ has an $A_1$-{\it point\/} 
(respectively, $A_2$-{\it point\/}) at $p$ if
$c$ at $p$ is regular (respectively, a cusp).
\begin{corollary}\label{cor:contsing}
Let\/ $f:U\to \R^3_1$ be an admissible frontal.
Then the following hold for\/ $k_1,k_2=2,3$.
\begin{itemize}
\item The curve\/ $\hat\bgamma$ and the osculating ellipse
have\/ $(k_1,k_2)$-point contact if and only if\/ the
osculating lightlike surface have an\/ $A_{k_1+1}$-point,
and the transversal lightlike surface have an $A_{k_2+1}$-point,
\item
the curve\/ $\hat\bgamma$ and the $N$-osculating parabola
have\/ $(k_1,k_2+1)$-point contact if and only if the
osculating lightlike surface have an\/ $A_{k_1+1}$-point,
and the
transversal lightcone pedal have an \/$A_{k_2+2}$-point,
\item
the curve\/ $\hat\bgamma$ and the $L$-osculating parabola
have\/ $(k_1+1,k_2)$-point contact if and only if the
osculating lightcone pedal have\/ $A_{k_1+2}$-point,
and
the transversal lightlike surface have\/ $A_{k_2+1}$-point,
\item
the curve\/ $\hat\bgamma$ and the tangent line 
have\/ $(k_1+1,k_2+1)$-point contact if and only if
osculating lightcone pedal have an\/ $A_{k_1+2}$-point,
and
transversal lightcone pedal have an\/ $A_{k_2+2}$-point.
\end{itemize}
\end{corollary}

\section{Special lightlike loci}\label{sec:special}
We consider a frontal $f$, 
where $f_L$ or $f_N$ has special properties.
Since $f_L|_{S(f_L)}=\hat\bgamma(u)+\bL(u)/\alpha_L$
(respectively, $f_N|_{S(f_N)}=\hat\bgamma(u)+\N(u)/\alpha_N$),
the singular value of $f_L$ (respectively, $f_N$),
is one point set if and only if
$\sigma_L\equiv0$
(respectively, $\sigma_N\equiv0$).
As we mentioned above, these conditions do not depend on the choice of $\bL$.
We set two constant points
$V_L=\hat\bgamma(u)+\bL(u)/\alpha_L$ if
$\sigma_L\equiv0$, and
$V_N=\hat\bgamma(u)+\N(u)/\alpha_N$ if
$\sigma_N\equiv0$,
these points do not depend on the choice of 
$\bL$ (cf. \eqref{eq:alphadepend}).
In this section, we consider
geometric meanings of
$\sigma_L$
and
$\sigma_N$.
Since the function 
\begin{align*}
&d_L(u)
=\left|
\hat\bgamma(u)-V_C
\right|
=\left|
\hat\bgamma(u)-\hat\bgamma(u)-\dfrac{\bL(u)}{\alpha_L}
\right|
\\
&\left(
\text{respectively,\ }
d_N(u)
=\left|
\hat\bgamma(u)-V_N
\right|
=\left|
\hat\bgamma(u)-\hat\bgamma(u)-\dfrac{\N(u)}{\alpha_N}
\right|
\right)
\end{align*}
vanishes identically, 
if $\sigma_L\equiv0$
(respectively, $\sigma_N\equiv0$),
the curve $\hat\bgamma$ lies on the lightcone $LC_L$ whose vertex is
$V_L$
(respectively, 
the lightcone $LC_N$ whose vertex is
$V_N$).
If 
$\sigma_L\equiv0$
and
$\sigma_N\equiv0$
holds simultaneously,
then $\hat\bgamma$ is an ellipse
of an intersection $LC_L\cap LC_N$.
Thus the pair
$(\sigma_L,\sigma_N)$
measures how $L(f)$ is close to a ellipse which is obtained as an
intersection of two lightcones.
This ellipse is a (Euclidean) circle
if and only if
$$
\inner{\pmt{1\\0\\0}}
{\dfrac{\bL(u)}{\alpha_L(u)}+\dfrac{\N(u)}{\alpha_N(u)}}=0.
$$
\begin{example}\label{eq:tangcone}
Let $lc(x,y)=(x, y, \sqrt{x^2 + y^2})$ be a lightcone,
and let 
$$\nu_l(x,u)=
\dfrac{1}{\sqrt{2}}\left(
-\dfrac{x}{\sqrt{x^2+y^2}},-\dfrac{y}{\sqrt{x^2+y^2}},1\right)
$$
be a Euclidean unit normal of $lc(x,y)$.
Let $\gamma(u)$ be a curve in the $xy$-plane,
$\hat\bgamma(u)=lc(\gamma(u))$,
and
let $\nu_\gamma(u)=\nu_l(\gamma(u))\times\hat\bgamma'(u)
/
|\nu_l(\gamma(u))\times\hat\bgamma'(u)\cdot
\nu_l(\gamma(u))\times\hat\bgamma'(u)|
$
be a Euclidean left-word unit normal vector of $\hat\bgamma$ 
as a curve on $lc$.
We set
$$
f_\gamma(u,v)=\hat\bgamma(u) + r 
\big(\nu_l(\gamma(u)) + \cos v \nu_l(\gamma(u)) + \sin v \nu_\gamma(u)\big)
$$
for $r>0$.
Then by the construction, $L(f_\gamma)$ is the image of $\hat\bgamma$,
and
the image of $(f_\gamma)_L$ is the image of $lc$. We set $r=1/10$.
\begin{enumerate}
\item Let us set $\gamma_1(u)=(\cos u, 2 \sin u)/3$.
Then the singular value of $(f_{\gamma_1})_L$ is a point,
and $(f_{\gamma_1})_L$ is a lightcone.
The images of $f_{\gamma_1}$ and $(f_{\gamma_1})_L$ are
drawn in Figure \ref{fig:tangcone}, left.
\item Let us set $\gamma_2(u)=(2 \sin u + 1, \sqrt3 \cos u)/2$.
Then both of the singular values of $(f_{\gamma_2})_L$ 
and $(f_{\gamma_2})_N$ are points,
and $(f_{\gamma_2})_L$ and $(f_{\gamma_2})_N$ are lightcones.
The axes of these lightcones do not coincide.
The images of $f_{\gamma_2}$, $(f_{\gamma_2})_L$ and $(f_{\gamma_2})_N$ are
drawn in Figure \ref{fig:tangcone}, center (i.e. $\gamma_2$ is the osculating ellipse).
\item Let us set $\gamma_3(u)=(\cos u, \sin u)$.
Then both of the singular values of $(f_{\gamma_3})_L$ 
and $(f_{\gamma_3})_N$ are points,
and $(f_{\gamma_3})_L$ and $(f_{\gamma_3})_N$ are lightcones.
Furthermore, the axes of these lightcones coincide.
The images of $f_{\gamma_3}$, $(f_{\gamma_3})_L$ and $(f_{\gamma_3})_N$ are
drawn in Figure \ref{fig:tangcone}, right (i.e. $\gamma_3$ is the osculating ellipse (circle in the Euclidean sense)).
\end{enumerate}
\begin{figure}[htb]
\begin{center}
\begin{tabular}{ccc}
 \includegraphics[width=0.3\linewidth]{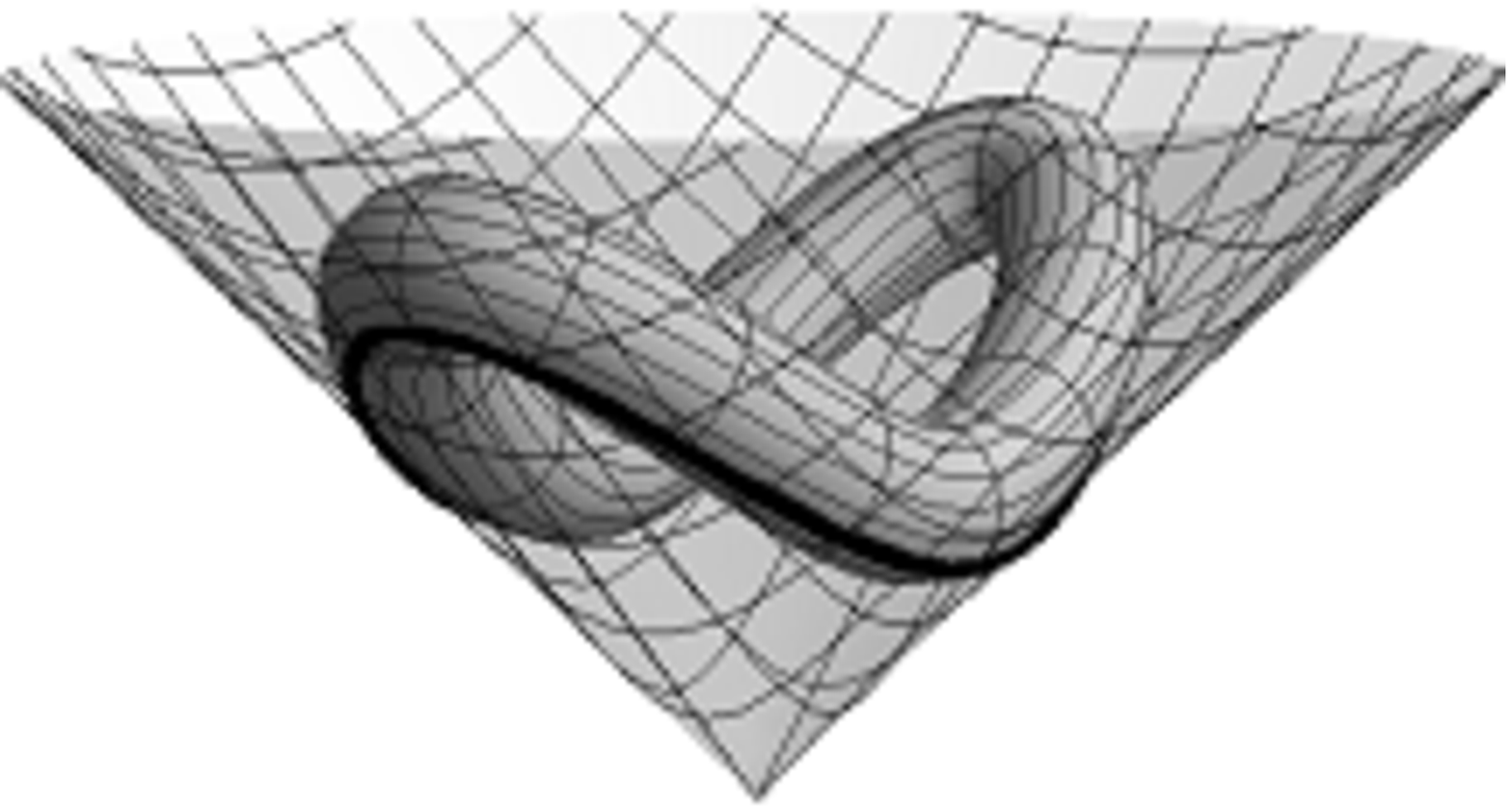}
&\includegraphics[width=0.3\linewidth]{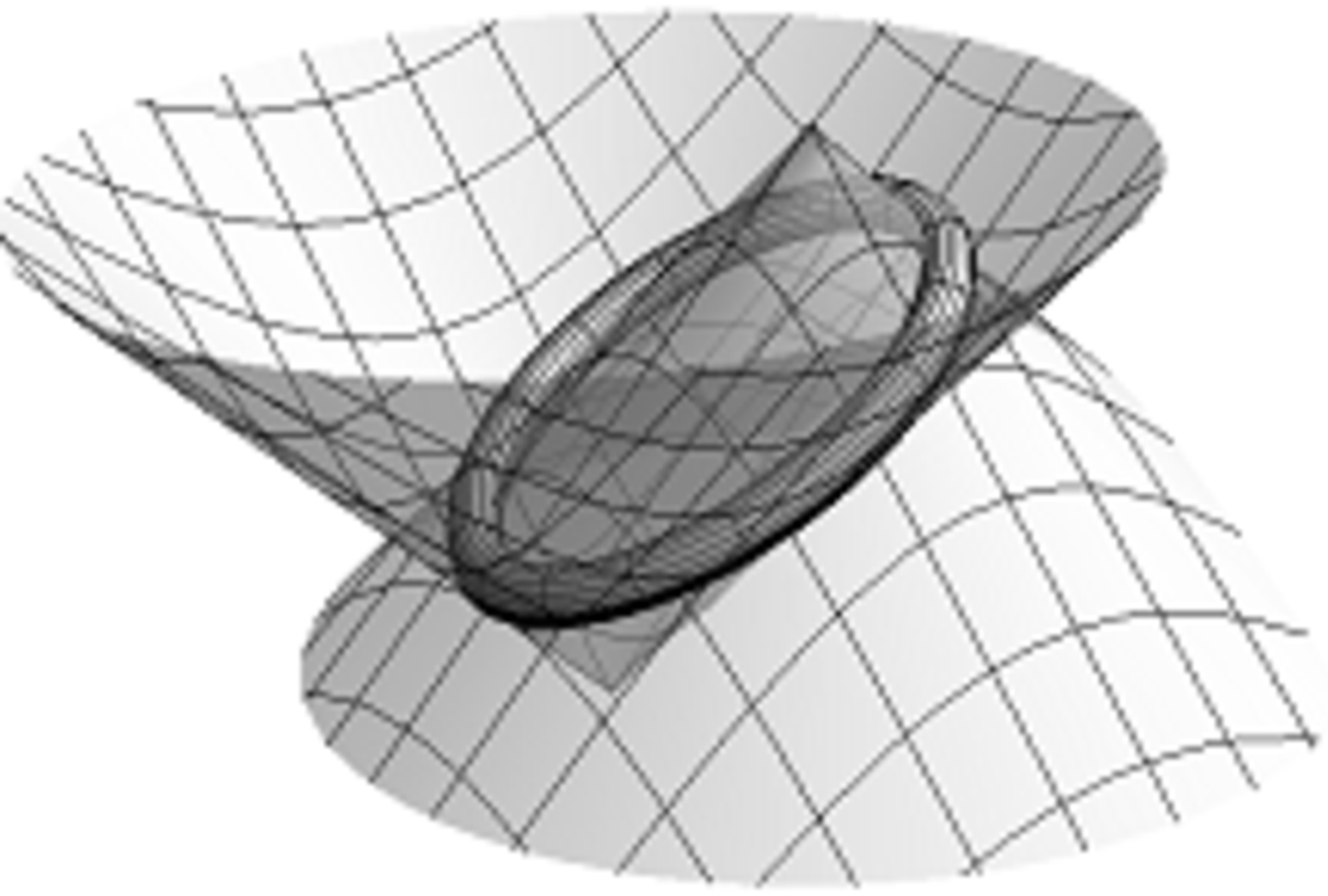}
&\includegraphics[width=0.27\linewidth]{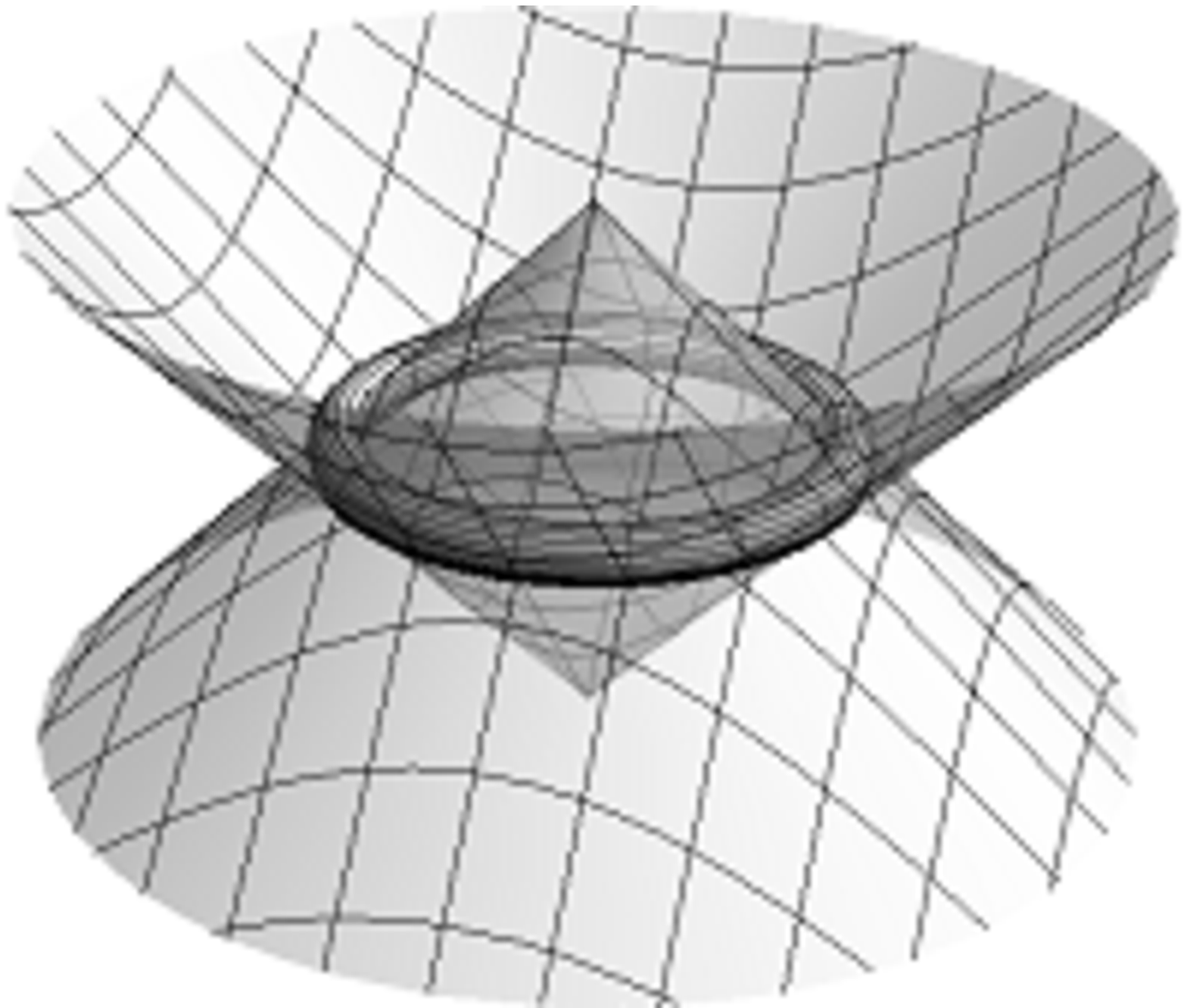}
\end{tabular}
\end{center}
\caption{Surfaces of Example \ref{eq:tangcone}.}
\label{fig:tangcone}
\end{figure}
\end{example}
\begin{example}\label{eq:passce}
We give an example $\sigma_L$ and $\alpha_N$ are constantly zero
and $L(f)$ pass the cuspidal edge but it is not along cuspidal edge.
Let $lc(x,y)$ is as in Example \ref{eq:tangcone}
and $lp(x,y)=(x, y, -(x - 1) - (y - 1))$ be a lightlike plane.
We set 
$\hat\bgamma(u)=2(1,\cos u,\sin u)/(1+\cos u+\sin u)$.
Then $\hat\bgamma(u)$ is a parametrization of the intersection
of the images of $lc$ and $lp$.
Let us set
$c(u,v)=(c_1(u,v),c_2(u,v))=
(v (2 u+v),v^2 (3 u+2 v))$,
and
$$
f(u,v)=\hat\bgamma(u)+c_1(u,v)\pmt{1\\1\\0}+c_2(u,v)\pmt{-1\\1\\0}.
$$
Then $f_v(u,u)=0$ and $f$ at $(u,u)$ is a cuspidal edge.
On the other hand, $f_v(u,0)=2u(1,1,0)$ holds, thus 
$f$ is lightlike at $f(u,0)=\hat\bgamma(u)$.
We set $\bL=(1,1,0)$ and $\N=(-1,1,0)$.
Then $\sigma_L=0$ and $\alpha_N=0$ hold.
In fact, by the above construction,
$\hat\bgamma$ is contained in the osculating contact lightcone
and the transversal contact lightlike plane 
(i.e. $\hat\bgamma$ is the $N$-osculating parabola).
\end{example}
\appendix
\section{Proof of Theorem \ref{thm:ghcont}}\label{sec:proofthm}
One can easily see that $g_L'(0)=g_L''(0)=0$.
We set a function $\beta_L(u)=\inner{\bL(u)}{\hat\bgamma(u)-\bx_L}$.
Since
\begin{align}
g_L^{(3)}&=
\inner{-2\alpha_N\alpha_L\be+\sigma_N\bL+\sigma_L\N}{\hat\bgamma-\bx_L}
\nonumber\\
&=
-2\alpha_N\alpha_L\inner{\be}{\hat\bgamma-\bx_L}
+\sigma_N\inner{\bL}{\hat\bgamma-\bx_L}
+\sigma_L\inner{\N}{\hat\bgamma-\bx_L}\nonumber\\
&=
-2\alpha_N\alpha_Lg_L'
+\sigma_N\beta_L
+\sigma_L\inner{\N}{\hat\bgamma-\bx_L},\label{eq:gdiff}
\end{align}
and $\beta_L(0)=0$, we see 
$g_L^{(3)}(0)=0$ if and only if $\sigma_L(0)=0$.
We have the following lemma.
\begin{lemma}\label{lem:ghzero}
If\/
$g_L^{(l)}(0)=0$, then\/
$\beta_L^{(l)}(0)=0$ ($l=1,\ldots,k$).
\end{lemma}
\begin{proof}
Since $\beta_L(0)=0$, 
the case $k=1$ follows from
\begin{equation}\label{eq:hdiff}
\begin{array}{rl}
\beta_L'
&=\inner{-\sigma_L\be-\alpha_G\bL}{\hat\bgamma-\bx_L}\\
&=
-\alpha_L\inner{\be}{\hat\bgamma-\bx_L}
-\alpha_G\inner{\bL}{\hat\bgamma-\bx_L}
=
-\alpha_Lg_L'
-\alpha_G\beta_L.
\end{array}
\end{equation}
We assume that the assertion is true for $k=1,\ldots,K$,
and we assume $g_L^{(l)}(0)=0$ $k=1,\ldots,K+1$.
Then by the assumption, $\beta_L^{(l)}(0)=0$ $k=1,\ldots,K$
holds.
Then $\beta_L^{K+1}(0)=0$ follows from $K$ times 
differentiation of \eqref{eq:hdiff}.
\end{proof}
\begin{proof}[Proof of Theorem {\rm \ref{thm:ghcont} (1)}]
We have shown the case $k=3$.
We assume that the assertion is true for $k=1,\ldots,K$.
We assume that $g_L^{(l)}(0)=0$ ($k=3,\ldots,K$), then
by the assumption of induction, $\sigma_L^{(l-3)}(0)=0$ ($l=3,\ldots,K$)
holds.
By Lemma \ref{lem:ghzero}, we have $\beta_L^{(l)}(0)=0$ ($k=3,\ldots,K$).
Then by $K-2$ times differentiation of \eqref{eq:gdiff},
we see
$$
g_L^{(K+1)}(0)=\sigma_L^{(K-2)}(0)\inner{\N(0)}{\hat\bgamma(0)-\bx_L}.
$$
Thus the assertion is true for $k=K+1$.
The assertion for $g_N$ and $\sigma_N$ can be shown by 
just interchanging the subscripts $N$ and $L$.
\end{proof}
One can easily see that $\widetilde h_L(0)=\widetilde h_L'(0)=0$.
We set a function $\delta_L(u)=\inner{\bL(u)}{\widetilde{\bL(0)}}$.
Since
\begin{align}
\widetilde h_L''
=&
\inner{\alpha_N\bL+\alpha_L\N}{\widetilde{\bL(0)}}\nonumber\\
=&
 \alpha_N\inner{\bL}{\widetilde{\bL(0)}}
+\alpha_L\inner{\N}{\widetilde{\bL(0)}}\nonumber\\
=&
 \alpha_N\delta_L
+\alpha_L\inner{\N}{\widetilde{\bL(0)}}\label{eq:hdiff2},
\end{align}
and $\delta_L(0)=0$, we see
$\widetilde h_L''(0)=0$ if and only if $\alpha_L(0)=0$.
We have the following  lemma.
\begin{lemma}\label{lem:tildehl}
If\/ $\alpha_L^{(l)}(0)=0$, then\/
$\delta_L^{(l+1)}(0)=0$ $(l=0,\ldots,k)$.
\end{lemma}
\begin{proof}
Since $\delta_L(0)=0$, 
the case $k=0$ follows from
\begin{equation}\label{eq:deltadiff}
\begin{array}{rl}
\delta_L'
&=\inner{-\alpha_L\be-\alpha_G\bL}{\widetilde{\bL(0)}}\\
&=
-\alpha_L\inner{\be}{\widetilde{\bL(0)}}
-\alpha_G\inner{\bL}{\widetilde{\bL(0)}}
=
-\alpha_L\inner{\be}{\widetilde{\bL(0)}}
-\alpha_G\delta_L.
\end{array}
\end{equation}
We assume that the assertion is true for $k=0,\ldots,K$,
and we assume $\alpha_L^{(l)}(0)=0$ $k=0,\ldots,K+1$.
Then by the assumption, $\delta_L^{(l+1)}(0)=0$ $k=1,\ldots,K$
holds.
Thus $\delta_L^{K+2}(0)=0$ follows from $K+1$ times 
differentiation of \eqref{eq:deltadiff}.
\end{proof}
\begin{proof}[Proof of Theorem {\rm \ref{thm:ghcont} (2)}]
We have shown the case $k=2$.
We assume that the assertion is true for $k=2,\ldots,K$.
We assume that $\widetilde h_L^{(l)}(0)=0$ ($k=2,\ldots,K$), then
by the assumption of induction, $\alpha_L^{(l-2)}(0)=0$ ($l=2,\ldots,K$)
holds.
By Lemma \ref{lem:tildehl}, we have $\delta_L^{(l-1)}(0)=0$ ($k=2,\ldots,K$).
Then by $K-1$ times differentiation of \eqref{eq:hdiff2},
we see
$$
\widetilde h_L^{(K+1)}(0)=
\alpha_L^{(K-1)}(0)\inner{\N(0)}{\widetilde{\bL(0)}}.
$$
Thus the assertion is true for $k=K+1$.
The assertion for $\widetilde h_N$ and $\sigma_N$ can be shown by 
just interchanging the subscripts $N$ and $L$.
\end{proof}


\begin{flushleft}
\textsc{Atsufumi Honda}\\
Department of Applied Mathematics, \\
Faculty of Engineering, \\
Yokohama National University,\\
Tokiwadai 79-5, Hodogaya, \\
Yokohama 240-8501, Japan\\
e-mail: honda-atsufumi-kp@ynu.ac.jp
\medskip

\textsc{Shyuichi Izumiya}\\ 
Department of Mathematics,\\
Hokkaido University,\\
Sapporo 060-0810, Japan\\
e-mail: izumiya@math.sci.hokudai.ac.jp
\medskip

\textsc{Kentaro Saji}\\
Department of Mathematics,\\
Kobe University,\\
Rokko 1-1, Nada, \\
Kobe 657-8501, Japan\\
e-mail: saji@math.kobe-u.ac.jp

\medskip
\textsc{Keisuke Teramoto}\\ 
Institute of Mathematics for Industry,\\ 
Kyushu University, \\
Motooka 744, Nishi-ku,\\
Fukuoka 819-0395, Japan\\
e-mail: k-teramoto@imi.kyushu-u.ac.jp

\end{flushleft}
\end{document}